\documentclass[10pt]{article}
\usepackage{mathrsfs}
\usepackage{amsfonts}
\usepackage{amsthm,amsmath,amssymb,anysize}
\newtheorem{lemma}{Lemma}[section]
\newtheorem{theorem}[lemma]{Theorem}

\newtheorem{proposition}[lemma]{Proposition}

\usepackage[numbers,sort&compress]{natbib}
\marginsize{3.6cm}{3.6cm}{3.6cm}{3cm}
\numberwithin{equation}{section}

\title{\textsf{ Hom-Lie superalgebra structures on exceptional simple Lie superalgebras of vector fields}}
\author{\textsc{Liping Sun$^{1,}$}\footnote{Supported by the NSF
  of  HLJ Province, China (A2015017)}\;\;
\textsc{and Wende
Liu$^{2,}$}\footnote{Corresponding author. Email:
\texttt{wendeliu@ustc.edu.cn}.\;\;Supported by the NNSF
  of China (11471090)}\;  \\
  \\
  \textit{$^{1}$Department of Applied sciences,}
  \textit{  Harbin University of Science and Technology} \\
  \textit{Harbin 150080, China}\\
\\
  \ \ \textit{$^{2}$School of Mathematical Sciences },
  \textit{Harbin Normal University}\\
  \textit{Harbin 150025, China}\\
  }
\date{ }
\begin{document}
\maketitle
\begin{quotation}
\noindent\textbf{Abstract:} In this paper, the Hom-Lie superalgebra structures on  exceptional simple Lie superalgebras of vector fields are studied.
Taking advantage of the $\mathbb{Z}$-grading structures and the transitivity, we prove that there is only the trivial Hom-Lie superalgebra structures on  exceptional simple Lie superalgebras. Our proof is obtained by studying the Hom-Lie superalgebra structures on their $0$-$\mathrm{th}$ and $-1$-$\mathrm{st}$ $\mathbb{Z}$-components.



  \end{quotation}

  \setcounter{section}{0}
\section{Introduction}
~~~~~~~~~

The motivations to study Hom-Lie algebras and related algebraic structures come from physics and deformations of Lie algebras, in particular, Lie algebras of vector fields. Recently, this kind of investigation becomes rather popular \cite{hls,ls1,ls2,ls3,2016,bm,M2,s,jl,xie}, due in part to the prospect of having a general framework in which one can produce many types of natural deformations of Lie algebras, in particular, $q$-deformations. The notion of Hom-Lie algebras was introduced by  J. T. Hartwig, D. Larsson and S. D. Silvestrov in \cite{hls} to describe the structures on certain deformations of Witt algebras and Virasoro algebras, which are applied widely in the theoretical physics\cite{wuli1,wuli2,wuli3,hu}. Later, W. J. Xie and Q. Q. Jin  gave a description of Hom-Lie algebra structures on semi-simple Lie algebras \cite{jl,xie}.
In 2010, F. Ammar
and A. Makhlouf  generalized Hom-Lie algebras to Hom-Lie
superalgebras \cite{M1}.
In 2013 and 2015, B. T. Cao, L. Luo, J. X. Yuan and W. D. Liu studied
 the Hom-structures on
finite dimensional  simple  Lie  superalgebras \cite{cl, yuan2}.

 As is well known, Kac  classified the infinite-dimensional simple linearly compact Lie superalgebras of vector fields, including  eight Cartan series and five exceptional simple Lie superalgebras \cite{K1}. In 2014, the Hom-Lie superalgebra structures on the eight infinite-dimensional Cartan series were investigated by  J. X. Yuan, L. P. Sun and W. D. Liu \cite{yuan}. In the present paper, we characterize
the  Hom-Lie superalgebra structures on  the five infinite-dimensional exceptional simple Lie superalgebras of vector fields.
Taking advantage of the $\mathbb{Z}$-grading structures and the transitivity, we analyse the Hom-Lie superalgebra structures  through computing the Hom-Lie superalgebra structures on the $0$-$\mathrm{th}$ and $-1$-$\mathrm{st}$ $\mathbb{Z}$-components. The main result of this study shows that the Hom-Lie superalgebra structures on the five infinite-dimensional exceptional simple Lie superalgebras of vector fields must be trivial.

\section{Preliminaries}
Throughout the paper, we will use the following notations.
$\mathbb{F}$ denotes an algebraiclly closed field of  characteristic zero. $\mathbb{Z}_2:=
\{\bar{0},\bar{1}\}$  is the additive group of two-elements.
The symbol $|x|$ and $\mathrm{zd(x)}$ denote the $\mathbb{Z}_2$-degree of a $\mathbb{Z}_2$-homogeneous  element $x$ and the $\mathbb{Z}$-degree of  a $\mathbb{Z}$-homogeneous element $x$, respectively. Notation $\langle v_1,\ldots, v_k\rangle$ is used to represent the linear space spanned by $v_1,\ldots,v_k$ over $\mathbb{F}.$
Now, let us review the five infinite-dimensional exceptional simple $\mathbb{Z}$-graded Lie superalgebras briefly. More detailed descriptions can be found in \cite{K1,SK}.

We construct a Weisfeiler filtration of $L$ by using open subspaces:
$$L=L_{-h}\supset L_{-h+1}\supset \ldots\supset L_0\supset L_1\supset\ldots$$
where $L$ is a simple linearly compact Lie superalgebra.
Let $\mathfrak{g}_i=L_i/L_{i+1},$  $\mathrm{Gr}L=\oplus^{\infty}_{i=-h}\mathfrak{g}_i.$ In \cite{K1}, all possible choices for nonpositive part $\mathfrak{g}_{\leq 0}=\oplus^{0}_{i=-h}\mathfrak{g}_i$ of the associated graded Lie superalgebra of $L$ were derived. The five exceptional simple Lie superalgebras are as follows(the subalgebra $\mathfrak{g}_{\leq 0}$ is written below as the $h+1$-tuple($\mathfrak{g}_{-h},\mathfrak{g}_{-h+1},\ldots \mathfrak{g}_{-1},\mathfrak{g}_0$)):
\begin{eqnarray*}
E(5,10):&& (\mathbb{F}^{5^*},\;\wedge^2(\mathbb{F}^5),\;\mathfrak{sl}(5))\\
E(3,6):&& (\mathbb{F}^{3^*},\;\mathbb{F}^3\otimes\mathbb{F}^2,\;\mathfrak{gl}(3)\oplus\mathfrak{sl}(2))\\
E(3,8):&&(\mathbb{F}^2,\;\mathbb{F}^{3^*},\;\mathbb{F}^3\otimes\mathbb{F}^2,\;\mathfrak{gl}(3)\oplus\mathfrak{sl}(2))\\
E(1,6):&& (\mathbb{F},\;\mathbb{F}^n,\; \mathfrak{cso}(n),n\geq 1,n \neq 2)\\
E(4,4):&& (\mathbb{F}^{4|4},\;\widehat{P}(4))
\end{eqnarray*}
 As we know, the five exceptional simple Lie superalgebras are transitive and irreducible. In particular, $\mathfrak{g}_{-i}=\mathfrak{g}_{-1}^i$ for $i\geq 1$ and the transitivity will be used frequently in this paper:
\begin{eqnarray}\label{e4}
\mathbf{transitivity}:
\;\mbox{if}~  x\in \mathfrak{g}_i,  i\geq 0, ~\mbox{then}~ [x,\mathfrak{g}_{-1}]=0~ \mbox{implies}~  x=0.
\end{eqnarray}

To study the Hom-Lie superalgebras structure, we recall the following definition \cite{cl}.

A  Hom-Lie superalgbra is a triple $(\mathfrak{g},[\cdot,\cdot],\sigma)$ consisting of a $\mathbb{Z}_2$-graded vector space $\mathfrak{g},$  an even bilinear mapping: $\mathfrak{g}\times \mathfrak{g} \rightarrow \mathfrak{g}$
and an even linear mapping: $\sigma:\mathfrak{g}\rightarrow \mathfrak{g}$ satisfying:
\begin{eqnarray}\label{e3}
&& \sigma[x,y]=[\sigma(x),\sigma(y)],\\
&& [x,y]=-(-1)^{|x||y|}[y,x],\\\label{e1}
&& (-1)^{|x||z|}[\sigma(x),[y,z]]+(-1)^{|y||x|}[\sigma(y),[z,x]]+(-1)^{|z||y|}[\sigma(z),[x,y]]=0. \label{e2}
\end{eqnarray}
for all homogeneous elements $x, y$, $z \in \mathfrak{g}.$

An even linear mapping $\sigma$ on a Lie superalgebra $\frak{g}$ is called a  Hom-Lie superalgebra structure on $\frak{g}$ if $(\frak{g}, [\cdot,\cdot], \sigma)$ is a  Hom-Lie superalgebra. In particular, $\sigma$ is called trivial if $\sigma=\mathrm{id}|_g.$

It is obvious that $\sigma$ is graded and if $\sigma$ is a  Hom-Lie superalgebra structure on a simple Lie superalgebra $\mathfrak{g}$, then $\sigma$ must be a monomorphism.
\begin{lemma}\label{00}
Let $\mathfrak{g}=\oplus_{i\geq-h}\mathfrak{g}_i$ be a transitive Lie superalgebra. If $\sigma$ is a  Hom-Lie superalgebra structure on  $\mathfrak{g}$ and
$\sigma|_{\mathfrak{g}_{-1}}=\mathrm{id}|_{\mathfrak{g}_{-1}},$ then $\sigma(x)-x\in \mathfrak{g}_{-k},$ where $x\in \mathfrak{g}_{0},$ $k\geq 1.$
\end{lemma}
\begin{proof}
For any  $y\in \mathfrak{g}_{-1},$ from $\sigma|_{\mathfrak{g}_{-1}}=\mathrm{id}|_{\mathfrak{g}_{-1}}$ and (\ref{e3}), one can deduce
$$[\sigma(x), y]=[\sigma(x),\sigma(y)]=\sigma[x,y]=[x,y].$$
Then $[\sigma(x)-x, y]=0.$ So, $[\sigma(x)-x, \mathfrak{g}_{-1}]=0.$  The transitivity(\ref{e4}) of $\mathfrak{g}$ and the gradation of $\sigma$ imply that $\sigma(x)-x\in \mathfrak{g}_{-k},$ $k\geq 1.$
\end{proof}
\begin{lemma}\label{l1}
Let $\mathfrak{g}=\oplus_{i\geq -h} \mathfrak{g}_i$ be a transitive and irreducible Lie superalgebra. If $\sigma$ is a  Hom-Lie superalgebra structure on $\mathfrak{g},$ and satisfies
$$\sigma|_{\mathfrak{g}_0\oplus \mathfrak{g}_{-1}}=\mathrm{id}|_{\mathfrak{g}_0\oplus \mathfrak{g}_{-1}},$$ then $$\sigma=\mathrm{id}|_{\mathfrak{g}}.$$
\end{lemma}
\begin{proof}
Let $i\geq 1.$ Equation (\ref{e3}) and $\mathfrak{g}_{-i}=\mathfrak{g}_{-1}^i$  imply
\begin{eqnarray}\label{l0}
\sigma\mid_{\mathfrak{g}_{-i}}=\mathrm{id}\mid_{\mathfrak{g}_{-i}}.
\end{eqnarray}

Suppose $x\in \mathfrak{g}_{i},$ $y, z\in \mathfrak{g}_{\leq 0}=\oplus^{0}_{i=-h}\mathfrak{g}_i,$  by (\ref{e2}) and (\ref{l0}), one can deduce
$$[\sigma(x),[y,z]]=[x,[y,z]].$$
Since
$\mathfrak{g}$ is transitive and irreducible, then $\mathfrak{g}_{-1}=[\mathfrak{g}_0,\mathfrak{g}_{-1}]$. Together with Equation (\ref{e1}), it implies
$$[\sigma(x)-x,\mathfrak{g}_{\leq 0}]=0.$$ Then $\sigma(x)-x=0,$
$\sigma|_{\mathfrak{g}_{i}}=\mathrm{id}|_{\mathfrak{g}_{i}}.$
Thus, we have
$\sigma=\mathrm{id}|_{\mathfrak{g}},$  that is,
$\sigma$ is trivial.
\end{proof}
According to Lemma \ref{l1}, to study the Hom-Lie superalgebra structures on  five exceptional simple  Lie superalgebras,
we can begin with their $0$-$\mathrm{th}$ and $-1$-$\mathrm{st}$ components. In the remaining of this paper, we will directly use Equation (\ref{e2}) without further remarks.

\section {Hom-Lie superalgebra structures on  five exceptional simple Lie superalgebras}
\subsection{E(4,4)}

  According to the studies reported in \cite{SK}, there is a unique  irreducible gradation over $E(4,4),$ such that $E(4,4)$ is the only inconsistent gradated algebra(its $-1$-$\mathrm{st}$ component  is not purely odd) of  the five exceptional simple Lie superalgebras. The $0$-$\mathrm{th}$ component $E(4,4)_{0}$ is isomorphic to $\widehat{\mathrm{P}}(4),$
 which is the unique nontrivial center extension of $\mathrm{P}(4).$ As $ E(4,4)_0$-module, $ E(4,4)_{-1}$ is isomorphic to the nature module $\mathbb{F}^{4|4}.$ In the followings, we will use the notations  of  $E(4,4)$, which are introduced in \cite{five}.

 Let $C=(c_{ij})\in \mathfrak{gl}_4(\mathbb{F})$ be a skew-symmetric matrix. $\widetilde{C}=(\widetilde{c}_{ij})=(\varepsilon_{ijkl}c_{kl})$ stands for the Hodge dual of $C,$ where $\varepsilon_{ijkl}$ is the symbol of permutation $(1234)\mapsto(ijkl).$
For short, write $i':=i+4,$ $i=1,2,3,4.$  Fix a basis of $\widehat{P}(4):$
$\mathcal{A}\cup \mathcal{B}\cup\mathcal{C}\cup \mathcal{I},$  where
\begin{eqnarray*}
&&\mathcal{A}:=\{E_{ij}-E_{j'i'},\quad 1\leq i\neq j\leq 4\},\\
&&\mathcal{B}:=\{E_{ij'}+E_{ji'},\quad 1\leq i\leq j\leq 4\},\\
&&\mathcal{C}:=\{E_{i'j}-E_{j'i}-(\widetilde{E}_{i'j}-\widetilde{E}_{j'i}),\quad 1\leq i<j\leq 4\},\\
&&\mathcal{I}:=\{I,\;E_{ii}-E_{i+1,i+1}-(E_{i'i'}-E_{i'+1,i'+1}),\quad 1\leq i\leq 3 \},
\end{eqnarray*}
and $I$ is the unite matrix.

Suppose $E(4,4)_{-1}=\langle v_i,v_{i'}|i=1,\ldots,4\rangle,$ where $|v_i|=\bar{0},$ $|v_{i'}|=\bar{1}.$
\begin{proposition}\label{p4-1}
If $\sigma$ is a Hom-Lie superalgebra structure on $E(4,4),$ then
$$\sigma|_{E(4,4)_{-1}}=\mathrm{id}|_{E(4,4)_{-1}}.$$
\end{proposition}
\begin{proof}
Let $i=1,\ldots,8,$  $k,j=1,\ldots,4$ and $i\neq k',$   $j\neq i.$ By using Equation (\ref{e2}), one can deduce
\begin{eqnarray} \label{E4-1}
[\sigma(v_i),v_k]=[\sigma(v_i),[E_{kj}-E_{j'k'},v_j]]=0.
\end{eqnarray}
Similarly, one has
\begin{eqnarray}\label{E4-2}
[\sigma(v_i),v_{k'}]&=&-[\sigma(v_i),[E_{jk}-E_{k'j'},v_{j'}]]=0
\end{eqnarray}
and
$$[\sigma(v_k),v_{k'}]=-[\sigma(v_k),[E_{k'j}-E_{j'k}-(\widetilde{E}_{k'j}-\widetilde{E}_{j'k}),v_{j}]]=-[\sigma(v_j),v_{j'}].$$
The arbitrariness of $k,j$  in the last equation shows
\begin{eqnarray}\label{E4-3}
[\sigma(v_i),v_{i'}]=0,\quad i=1,\ldots,4.
\end{eqnarray}
It follows from Equations (\ref{E4-1})-(\ref{E4-3}) that
$$
[\sigma(E(4,4)_{-1}),E(4,4)_{-1}]=0.
$$
Using the transitivity (\ref{e4}), one gets $\sigma(E(4,4)_{-1})\subseteq E(4,4)_{-1}.$
Recall that $|\sigma|=\bar{0},$ one may suppose
$$\sigma(v_i)=\sum_{m=1}^8\lambda_mv_m,\lambda_{m},\in \mathbb{F}.$$
Now, for distinct $i,j,k,l=1,2,3,4,$ suppose $(1234)\mapsto(ijkl)$ is an even permutation and
$$c=E_{k'j}-E_{j'k}-(E_{li'}-E_{il'})\in \mathcal{C},\;
a=E_{jl}-E_{l'j'}\in \mathcal{A}.$$
Obviously, $[a,v_i]=0,$ $[v_i,c]=0.$ Then
\begin{eqnarray*}
0=[\sigma(v_i),[c,a]]
=[\sigma(v_i),E_{k'l}-E_{l'k}-(E_{ij'}-E_{ji'})]
=-\lambda_lv_{k'}+\lambda_kv_{l'}.
\end{eqnarray*}
Hence, $\lambda_k=\lambda_l=0$, and then $\sigma(v_i)=\lambda_iv_i,$ $i=1,2,3,4.$

The equation
\begin{eqnarray*}
0=[\sigma(v_{i'}),[E_{kj}-E_{j'k'},E_{jl'}+E_{lj'}]]=[\sigma(v_{i'}),E_{kl'}+E_{lk'}]=\lambda_{l'}v_{k}+\lambda_{k'}v_{l},
\end{eqnarray*}
implies $\lambda_{k'}=\lambda_{l'}=0.$  One has $\sigma(v_{i'})=\lambda_{i'}v_{i'}.$

At last, let us prove $\sigma(v_i)=v_i$ and $\sigma(v_{i'})=v_{i'}.$
For  distinct $i,j,k,$ put $$h=E_{jj}-E_{j+1,j+1}-(E_{kk}-E_{k+1,k+1})\in\mathcal{I},\;a=E_{ji}-E_{j'i'}\in\mathcal{A}.$$
Clearly,
\begin{eqnarray}\label{E4-4}
[\sigma(h),v_{j}]=[\sigma(h),[a,v_{i}]]=
-[\sigma(v_{i}),a]=
\lambda_iv_j.
\end{eqnarray}
Recall that $\sigma$ is monomorphic, one may suppose $\sigma^{-1}$ is a left linear inverse of $\sigma,$ then
$$\sigma^{-1}[\sigma(h),v_{j}]=[h,\sigma^{-1}(v_j)]=\lambda_j^{-1}[h,v_j]=\lambda_j^{-1}v_j.$$
It implies that
\begin{eqnarray}\label{E4-5}
[\sigma(h),v_j]=\sigma(\lambda_j^{-1}v_j)=v_j.
\end{eqnarray}
Comparing (\ref{E4-4}) with (\ref{E4-5}), one gets $\lambda_i=1.$ Analogously, one may check $\sigma(v_{i'})=v_{i'}.$

Summing up the above, we have proved $\sigma|_{E(4,4)_{-1}}=\mathrm{id}|_{E(4,4)_{-1}}.$
\end{proof}

\begin{proposition}\label{p4-2}
If $\sigma$ is a  Hom-Lie superalgebra structure on $E(4,4),$ then
$$\sigma|_{E(4,4)_{0}}=\mathrm{id}|_{E(4,4)_{0}}.$$
\end{proposition}
\begin{proof}
Put $x\in E(4,4)_0,$ $v\in E(4,4)_{-1}.$ Using Propersition \ref{p4-1} and Lemma \ref{00}, one can deduce
 $\sigma(x)-x\in E(4,4)_{-1}.$
Note that $\mathcal{A}$ and $\mathcal{I}$ are generated(Lie product) by
$\mathcal{B}$ and $\mathcal{C},$ it is sufficient to prove the cases for $x\in \mathcal{B}$ and $x\in \mathcal{C}.$

$\mathbf{Case\ 1:}$
Let $x=E_{ij'}+E_{ji'}\in\mathcal{B}.$ Noting that $|\sigma|=\bar{0},$ one may suppose $$\sigma(x)=x+\sum_{m=1}^4\lambda_{m'}v_{m'},\quad\lambda_{m'}\in \mathbb{F}.$$
Put $$c=E_{k'l}-E_{l'k}-(\widetilde{E}_{k'l}-\widetilde{E}_{l'k})\in\mathcal{C},\quad b=E_{kk'}\in\mathcal{B},$$
where $k,l\neq i,j.$
 Then
\begin{eqnarray*}
0=[\sigma(x),[b,c]]=\left[E_{ij'}+E_{ji'}+\sum_{m=1}^4\lambda_{m'}v_{m'},E_{kl}-E_{l'k'}\right]=\lambda_{k'}v_{l'}.
\end{eqnarray*}
It implies that $\lambda_{k'}=0,$ $k\neq i,j.$ Then,
\begin{eqnarray}\label{E4-6}
\sigma(x)=x+\lambda_{i'}v_{i'}+\lambda_{j'}v_{j'}.
\end{eqnarray}
Now, put
$$
c=E_{i'k}-E_{k'i}-(\widetilde{E}_{i'k}-\widetilde{E}_{k'i'})\in\mathcal{C},\quad b=E_{ii'}\in\mathcal{B},\quad k\neq i,j.
$$
Using Equation (\ref{E4-6}), one may suppose $\sigma(b)=b+\mu_{i'}v_{i'}, \mu_{i'}\in \mathbb{F},$ then
\begin{eqnarray*}
0=[\sigma(x),a]=[E_{ij'}+E_{ji'}+\lambda_{i'}v_{i'}+\lambda_{j'}v_{j'},E_{ik}-E_{k'i'}]=\lambda_{i'}v_{k'}.
\end{eqnarray*}
So $\lambda_{i'}=0.$  That is, $\sigma(x)=x+\lambda_{j'}v_{j'}.$
Similarly, put
$$c=E_{j'k}-E_{k'j}-(\widetilde{E}_{j'k}-\widetilde{E}_{k'j'})\in\mathcal{C},\quad b=E_{jj'}\in\mathcal{B}.$$
One may obtain $\lambda_{j'}=0.$ Thus, we proved  $\sigma(x)=x$ for any $x\in \mathcal{B}.$

$\mathbf{Case\ 2:}$
Let $x=E_{i'j}-E_{j'i}-(E_{kl'}-E_{lk'})\in\mathcal{C},$ where $\varepsilon_{ijkl}=1.$ Noting that $|\sigma|=\bar{0},$ one may assume  $\sigma(x)=x+\sum_{m=1}^4\mu_{m'}v_{m'},$ $\mu_{m'}\in\mathbb{F}.$

Firstly, put $$c=E_{k'l}-E_{l'k}-(E_{ji'}-E_{ij'})\in\mathcal{C},\quad b=E_{kk'}\in\mathcal{B},$$
where $\varepsilon_{ijkl}=1.$
Equation (\ref{e2}) and the result $\sigma(b)=b$ obtained in Case 1 imply
$$0=[\sigma(x),[c,b]]=[x+\sum_{m=1}^4\mu_{m'}v_{m'},E_{kl}-E_{l'k'}]=c_0+\mu_{k'}v_{l'},\quad c_0\in\mathcal{C}.$$
The equation above shows $\mu_{k'}=0.$  By the arbitrariness of $k\neq i,j,$ one has
\begin{eqnarray}\label{E4-7}
\sigma(x)=x+\mu_{i'}v_{i'}+\mu_{j'}v_{j'}.
\end{eqnarray}

Secondly, put $$c=E_{l'j}-E_{j'l}-(E_{ik'}-E_{k'i})\in\mathcal{C},\quad b=E_{jj'}.$$
 Using Equation (\ref{E4-7}), one may suppose
$$\sigma(c)=c+\gamma_{l'}v_{l'}+\gamma_{j'}v_{j'},\quad \gamma_{l'},\gamma_{j'}\in\mathbb{F}.$$  On one hand,
\begin{eqnarray*}
[\sigma(x),[b,c]]=-[\sigma(b),[c,x]]-[\sigma(c),[x,b]]=E_{l'i}-E_{i'l}-(E_{kj'}-E_{jk'})-\gamma_{j'}v_{i'}.
\end{eqnarray*}
On the other hand,
\begin{eqnarray*}
[\sigma(x),[b,c]]&=&[E_{i'j}-E_{j'i}-(E_{kl'}-E_{lk'})+\mu_{i'}v_{i'}+\mu_{j'}v_{j'},E_{l'j'}-E_{jl}]\\
&=&E_{l'i}-E_{i'l}-(E_{kj'}-E_{jk'})-\mu_{j'}v_{l'}.
\end{eqnarray*}
Since $i\neq l,$ one has $\gamma_{j'}=\mu_{j'}=0.$  Then $\sigma(x)=x+\lambda_{i'}v_{i'}.$

At last, put
$$c=E_{i'k}-E_{k'i}-(E_{lj'}-E_{jl'})\in\mathcal{C},\quad b=E_{ii'}\in\mathcal{B}.$$
We can obtain $\lambda_{i'}=0$ in the same way. Thus,  $\sigma(x)=x$ is proved.

Summing up, we have proved  $\sigma(x)=x$ for any $x\in E(4,4)_0,$ that is, $\sigma|_{E(4,4)_{0}}=\mathrm{id}|_{E(4,4)_0}.$
\end{proof}

\subsection{ E(3,6), E(5,10) and E(3,8)}
Before studying the Hom-Lie superalgebra structures on $E(3,6), E(5,10)$ and $ E(3,8),$  we would like to review their algebraic structures briefly
 \cite{K1,SK,E36,E38}.
As we know in \cite{K1}, the even part $E(5,10)_{\bar{0}}$ of $E(5,10)$ is isomorphic to the Lie algebra $S_5$, which is of divergence 0 polynomial vector fields on $\mathbb{F}^5,$  i.e., polynomial vector fields annihilating the volume form $\mathrm{d}x_1\wedge \mathrm{d}x_2\wedge \ldots \wedge \mathrm{d}x_5.$ As $S_5$-module, the odd part $E(5,10)_{\bar{1}}$ of $E(5,10)$ is isomorphic to $ \mathrm{d}\Omega^1(5),$ the space of closed polynomial differential 2-forms on $\mathbb{F}^5.$
Next, we keep the notations in \cite{E36}:
$$\mathrm{d}_{ij}:=\mathrm{d}x_i\wedge \mathrm{d}x_j, \quad \partial_i:=\partial/\partial x_i.$$
One may denote  $D\in E(5,10)_{\bar{0}}$ through
$$D=\sum_{i=1}^5f_i\partial_i,\quad \mbox{where}\; f_i\in \mathbb{F}[[x_1,x_2,\dots, x_5]], \quad \sum_{i=1}^5\partial_i(f_i)=0,$$
and denote $E\in E(5,10)_{\bar{1}}$ through
\begin{eqnarray}\label{E5-x1}
E=\sum_{i,j=1}^5f_{ij}\mathrm{d}_{ij},\quad \mbox{where}\; f_{ij}\in \mathbb{F}[[x_1,x_2,\dots, x_5]], \quad \mathrm{d}E=0.
\end{eqnarray}
The bracket in $E(5,10)_{\bar{1}}$ is defined by
$$[f\mathrm{d}_{ij},g\mathrm{d}_{kl}]=\varepsilon_{tijkl}fg\partial_t,$$
where $\varepsilon_{tijkl}$ is the sign of permutation $(tijkl)$ when $\{tijkl\}=\{12345\}$ and zero otherwise.
The bracket of  $E(5,10)_{\bar{0}}$ with $E(5,10)_{\bar{1}}$ is defined by the usual action of vector fields on differential forms.
The irreducible consistent $\mathbb{Z}$-gradation over $E(5,10)$ is defined by letting  (see \cite{E36})
$$\mathrm{zd}(x_i)=2,\quad \mathrm{zd}(\mathrm{d})=-\frac{5}{2},\quad \mathrm{zd}(\mathrm{d}x_i)=-\frac{1}{2}.$$
Then  $E(5,10)=\oplus_{i\geq -2}\mathfrak{g}_i,$
where
\begin{eqnarray*}
&&\mathfrak{g}_0 \simeq \mathfrak{sl}(5)=\langle x_i\partial_j,\; x_k\partial_k-x_{k+1}\partial_{k+1}\;|\;i,j=1,\ldots, 5, i\neq j,\;k=1,\ldots,4\rangle,\\
&&\mathfrak{g}_{-1}\simeq \wedge^2\mathbb{F}^5=\langle\mathrm{d}_{ij}\;|\;1\leq i< j\leq 5\rangle,\\
&&\mathfrak{g}_{-2}\simeq  \mathbb{F}^{5^*}=\langle \partial_i \;|\; i=1,\ldots,5\rangle.
\end{eqnarray*}

The exceptional simple Lie superalgebra $E(3,6)$ is a subalgebra of  $E(5,10).$  The irreducible consistent $\mathbb{Z}$-gradation over $E(3,6)$ is induced by the  $\mathbb{Z}$-gradation  of $E(5,10)$ above. Let $h_1=x_1\partial_1-x_2\partial_2,$ $ h_2=x_2\partial_2-x_3\partial_3,$ $h_3=x_4\partial_4-x_5\partial_5,$ $h_4=-x_2\partial_2-x_3\partial_3+2x_5\partial_5.$ Then
$E(3,6)=\oplus_{i\geq -2}\mathfrak{g}_i,$
where
\begin{eqnarray*}
\mathfrak{g}_0 &\simeq &\mathfrak{gl}(3)\oplus\mathfrak{sl}(2)\\
& =&\langle x_i\partial_j,\; x_k\partial_l,\;h_m\;|\;i,j=1,2,3, i\neq j,\;k,l=4,5,\; k\neq l,\;m=1,\ldots,4\rangle,\\
\mathfrak{g}_{-1}&\simeq & \mathbb{F}^3\otimes \mathbb{F}^2=\langle\mathrm{d}_{ij}\;|\;i=1,2,3,\;  j=4,5\rangle,\\
\mathfrak{g}_{-2}&\simeq & \mathbb{F}^{3^*}=\langle \partial_i \;| \;i=1,2,3\rangle.
\end{eqnarray*}

The  exceptional simple Lie superalgebra $E(3,8)$, which is strikingly similar to $E(3,6),$ carries a unique irreducible consistent $\mathbb{Z}$-gradation\cite{E38} defined by
$$\mathrm{zd}(x_i)=-\mathrm{zd}(\partial_i)=\mathrm{zd}(\mathrm{d}x_i)=2,\; i=1,2,3;\;\quad \mathrm{zd}(x_4)=\mathrm{zd}(x_5)=-3.$$ Then $E(3,8)=\oplus_{i\geq -3}\mathfrak{g}_i,$
where
\begin{eqnarray}\label{E36-E38}
\mathfrak{g}_j=E(3,6)_j,\quad  j=0,-1,-2; \quad \mathfrak{g}_{-3}\simeq  \mathbb{F}^{2}\simeq \langle \mathrm{d}x_4, \mathrm{d}x_5\rangle.
\end{eqnarray}

Hereafter,  $\mathfrak{g}$  denotes $E(5,10),$ $E(3,6)$ or $E(3,8)$  unless otherwise noted.
We establish a technical lemma, which can be verified directly.

\begin{lemma}\label{E5-l1}
If $\sigma$ is a Hom-Lie superalgebra structure on $\mathfrak{g},$ then
\begin{itemize}
\item[$\mathrm{(1)}$]
$
[\sigma(\mathrm{d}_{ij}),\mathrm{d}_{il}]=\left\{\begin{array}{ll}
0,&\quad\mbox{if}\quad \mathrm{d}_{ij},\mathrm{d}_{il}\in E(5,10),\\
-[\sigma(\mathrm{d}_{il}),\mathrm{d}_{ij}],&\quad\mbox{if}\quad
\mathrm{d}_{ij},\mathrm{d}_{il}\in E(3,6)\mbox{or }E(3,8);
\end{array}\right.
$
\item[$\mathrm{(2)}$]
$[\sigma(\mathrm{d}_{ij}),\mathrm{d}_{kj}]=0,\;$ in particular, $[\sigma(\mathrm{d}_{ij}),\mathrm{d}_{ij}]=0;$
\item[$\mathrm{(3)}$]
$[\sigma(\mathrm{d}_{ij}),\mathrm{d}_{kl}]=[\mathrm{d}_{ij},\sigma(\mathrm{d}_{kl})]$ for distinct$\;i,j,k,l.$
\end{itemize}
\end{lemma}

\begin{proposition}\label{E5-p1}
Let $\mathfrak{g}$ be Lie superalgebra $E(5,10),$ $E(3,6)$ or $E(3,8).$ If $\sigma$ is a Hom-Lie superalgebra structure on $\mathfrak{g},$ then $$\sigma|_{\mathfrak{g}_{-1}}=\mathrm{id}|_{\mathfrak{g}_{-1}}.$$
\end{proposition}
\begin{proof}
First, let us prove $\sigma|_{\mathfrak{g}_{-1}}=\lambda\mathrm{id}|_{\mathfrak{g}_{-1}},$£¬where $\lambda\in \mathbb{F}.$

$\mathbf{Case\ 1:}$  $\;\mathfrak{g}=E(5,10).$
Noting  that the gradation over $E(5,10)$ is consistent  and $|\sigma|=\bar{0},$ $|\mathrm{d}_{ij}|=\bar{1},$ one may suppose
$$\sigma(\mathrm{d}_{ij})=\sum_{1\leq m<n\leq 5}f_{mn}\mathrm{d}_{mn},\quad\mbox{where}\; f_{mn}\in\mathbb{F}[[x_1,\ldots,x_5]].$$
By Lemma \ref{E5-l1} $(\mathrm{1}),$ one has
\begin{eqnarray*}
[\sigma(\mathrm{d}_{ij}),\mathrm{d}_{il}]
&=&\left[\sum_{1\leq m<n\leq 5}f_{mn}\mathrm{d}_{mn},\mathrm{d}_{il}\right]
=\left[\sum_{m,n\neq i,l}f_{mn}\mathrm{d}_{mn},\mathrm{d}_{il}\right]\\
&=&\sum_{m,n\neq i,l}[f_{mn}\mathrm{d}_{mn},\mathrm{d}_{il}]
=\sum_{m,n\neq i,l}\varepsilon_{qmnil}f_{mn}\partial_q=0.
\end{eqnarray*}
 The arbitrariness of $l$ shows
$$f_{mn}=0,\quad \mbox{where}\; m,n\neq i.$$
Similarly, by Lemma \ref{E5-l1} $(\mathrm{2}),$ one gets
$$f_{mn}=0,\quad \mbox{where}\; m,n\neq j.$$
Thus,$$\sigma(\mathrm{d}_{ij})=f_{ij}\mathrm{d}_{ij}.$$
In view of $\sigma(\mathrm{d}_{ij})\in E(5,10)_{\bar{1}}\simeq\mathrm{d}\Omega^1(5) $ and (\ref{E5-x1}), one has
$$0=\mathrm{d}(\sigma(\mathrm{d}_{ij}))=\mathrm{d}(f_{ij}\mathrm{d}_{ij})=\sum_{m=1}^5\partial_m(f_{ij})\mathrm{d}x_m\wedge\mathrm{d}x_{i}\wedge\mathrm{d}x_{j}.$$
Therefore,  $\partial_m(f_{ij})=0$ for $m\neq i,j,$ that is $f_{ij}\in \mathbb{F}[[x_i,x_j]].$
So one may suppose $\sigma(\mathrm{d}_{kl})=f_{kl}\mathrm{d}_{kl},$ where $f_{kl}\in \mathbb{F}[[x_k,x_l]].$
Then
$$[\sigma(\mathrm{d}_{ij}),\mathrm{d}_{kl}]=[f_{ij}\mathrm{d}_{ij},\mathrm{d}_{kl}]=\varepsilon_{qijkl}f_{ij}\partial_q,$$
$$[\sigma(\mathrm{d}_{kl}),\mathrm{d}_{ij}]=[f_{kl}\mathrm{d}_{kl},\mathrm{d}_{ij}]=\varepsilon_{qklij}f_{kl}\partial_q.$$
For distinct $i,j,k,l,$  Lemma \ref{E5-l1} $(\mathrm{3})$ implies
$f_{ij}=f_{kl}\in\mathbb{F}.$ Noting that there exists $t$ such that $t\neq i,j,k,l,$  one may obtain $f_{kt}=f_{ij}$ in the same way. Then $f_{ij}=\lambda \in \mathbb{F}$ for any $i,j=1,2,3,4,5$ and
$i\neq j.$
Thus, we proved $\sigma(\mathrm{d}_{ij})=\lambda\mathrm{d}_{ij}.$

$\mathbf{Case\ 2:}$ $\;\mathfrak{g}=E(3,6)$ or $E(3,8).$
 Noting that the gradation over $\mathfrak{g}$ is consistent and $|\sigma|=\bar{0},$ $|\mathrm{d}_{ij}|=\bar{1},$ one may suppose
$$\sigma(\mathrm{d}_{ij})=\sum_{m,n}f_{mn}\mathrm{d}_{mn},\quad\mbox{where}\; m=1,2,3,n=4,5, \;f_{mn}\in\mathbb{F}[[x_1,\ldots,x_5]].$$
By Lemma \ref{E5-l1} $(\mathrm{2}),$ one can deduce
\begin{eqnarray*}
\left[\sigma(\mathrm{d}_{ij}),\mathrm{d}_{kj}\right]=\left[\sum_{m,n}f_{mn}\mathrm{d}_{mn},\mathrm{d}_{kj}\right]=\left[\sum_{m\neq k ,n\neq j}f_{mn}\mathrm{d}_{mn},\mathrm{d}_{kj}\right]=0.
\end{eqnarray*}
It implies $f_{mt}=0,$ $m\neq k,$ $t\neq j.$
By the  arbitrariness of $k,$  one may suppose
$$\sigma(\mathrm{d}_{ij})=\sum_{m=1}^3f_{mj}\mathrm{d}_{mj},\;
\sigma(\mathrm{d}_{il})=\sum_{m=1}^3f_{ml}\mathrm{d}_{ml},\quad f_{mj}, f_{ml}\in\mathbb{F}[[x_1,\ldots,x_5].$$
Noting that $[\sigma(\mathrm{d}_{il}),\mathrm{d}_{kj}]=[\sigma(\mathrm{d}_{il}),[x_k\partial_i,\mathrm{d}_{ij}]]=-[\sigma(\mathrm{d}_{ij}),\mathrm{d}_{kl}],$ the simple computations show
$f_{mj}=f_{ml}$ for $ m=1,2,3.$
In view of $\sigma(\mathrm{d}_{ij})\in \mathfrak{g}_{\bar{1}}\subseteq\mathrm{d}\Omega^1(5)$ and (\ref{E5-x1}),
one could further deduce
\begin{eqnarray}\label{E3-0}
f_{mj}=f_{ml}\in\mathbb{F}[[x_1,x_2,x_3]].
\end{eqnarray}
For distinct $i,k,q=1,2,3,$
\begin{eqnarray*}
0&=&[\sigma(\mathrm{d}_{ij}),[x_i\partial_q,x_q\partial_k]]=[\sigma(\mathrm{d}_{ij}),x_i\partial_k]=\left[\sum_{m=1}^3f_{mj}\mathrm{d}_{mj},x_i\partial_k\right]\\
&=&-\sum_{m=1}^3\big(x_i\partial_k(f_{mj})\mathrm{d}_{mj}+f_{mj}x_i\partial_k(\mathrm{d}_{mj})\big)\\
&=&-(x_i\partial_k(f_{ij})+f_{kj})\mathrm{d}_{ij}-x_i\partial_k(f_{kj})\mathrm{d}_{kj}-x_i\partial_k(f_{qj})\mathrm{d}_{qj}.
\end{eqnarray*}
Therefore,
\begin{eqnarray}\label{E3-2}
f_{kj}=-x_i\partial_k(f_{ij}).
\end{eqnarray}
\begin{eqnarray}\label{E3-1}
\partial_{k}(f_{kj})=\partial_{k}(f_{qj})=0,
\end{eqnarray}
From (\ref{E3-0})-(\ref{E3-1}) and the arbitrariness of $k,$ one has
$f_{mj}=0,$ $m\neq i.$ Thus, one may suppose
$$\sigma(\mathrm{d}_{ij})=f_i\mathrm{d}_{ij},\; \sigma(\mathrm{d}_{kl})=f_k\mathrm{d}_{kl},\quad\mbox{where}\; f_{i}\in\mathbb{F}[[x_i]],\;f_k\in\mathbb{F}[[x_k]].$$
Then, for distinct $i,j,k,l,q,$
$$[\sigma(\mathrm{d}_{ij}),\mathrm{d}_{kl}]=[f_i\mathrm{d}_{ij},\mathrm{d}_{kl}]=\varepsilon_{qijkl}f_i\partial_q,$$
$$[\mathrm{d}_{ij},\sigma(\mathrm{d}_{kl})]=[\mathrm{d}_{ij},f_k\mathrm{d}_{kl},]=\varepsilon_{qijkl}f_k\partial_q.$$
By Lemma \ref{E5-l1} (3),  $f_i=f_k\in \mathbb{F}.$ Thus, we have proved $\sigma(\mathrm{d}_{ij})=\lambda\mathrm{d}_{ij}$ for some $\lambda\in\mathbb{F}.$

Now, let us prove $\lambda=1$ for $\sigma|_{\mathfrak{g}_{-1}}=\lambda\mathrm{id}|_{\mathfrak{g}_{-1}}.$
Suppose $x=x_i\partial_l,$ $y=x_l\partial_q,$ $z=\mathrm{d}_{qj}.$   Then
\begin{eqnarray*}
[\sigma(x),\mathrm{d}_{lj}]=[\sigma(x),[y,z]]
=-[\sigma(y),[z,x]]-[\sigma(z),[x,y]]
=\lambda\mathrm{d}_{ij}.
\end{eqnarray*}
Noting that $\sigma$ is a monomorphism, one may suppose $\sigma^{-1}$ is a left linear inverse of $\sigma.$ Then the equation above implies
$$\sigma^{-1}[\sigma(x),\mathrm{d}_{lj}]=[x,\lambda^{-1}\mathrm{d}_{lj}]=\lambda^{-1}\mathrm{d}_{ij}=\mathrm{d}_{ij}.$$
Thus, $\lambda=1,$ that is, $\sigma(\mathrm{d}_{ij})=\mathrm{d}_{ij}.$ The proof is completed.
\end{proof}
\begin{proposition}\label{E5-p2}
Let $\mathfrak{g}$ be Lie superalgebra $E(5,10),$ $E(3,6)$ or $E(3,8).$ If $\sigma$ is a   Hom-Lie superalgebra structure on $\mathfrak{g},$ then $$\sigma|_{\mathfrak{g}_{0}}=\mathrm{id}|_{\mathfrak{g}_{0}}.$$
\end{proposition}
\begin{proof}
$\mathbf{Case\ 1:}$ $\;\mathfrak{g}=E(5,10).\;$
Note that $E(5,10)_0$ is generated(Lie product) by $\{x_i\partial_j|i,j=1,\ldots, 5,\;i\neq j\},$  it is sufficient to prove $\sigma(x)=x$ only for $x=x_i\partial_j.$
 By Proposition \ref{E5-p1}, Lemma \ref{00} and $|\sigma|=\bar{0},$  we have $\sigma(x_i\partial_j)-x_i\partial_j\in \mathfrak{g}_{-2}.$  Then one may suppose
$$\sigma(x_i\partial_j)=x_i\partial_j+\sum_{m=1}^5\lambda_m\partial_m,\quad \lambda_m\in \mathbb{F}.$$
For distinct $i,j,q,k,l,$
$$0=[\sigma(x_i\partial_j),[x_q\partial_k,x_k\partial_l]]=\left[x+\sum_{m=1}^5 \lambda_m\partial_m,x_q\partial_l\right]=\lambda_q\partial_l.$$
Hence, $\lambda_q=0.$ The arbitrariness of $q\neq i,j$ shows
$$\sigma(x_i\partial_j)=x_i\partial_j+\lambda_i\partial_i+\lambda_j\partial_j.$$
Similarly,
$$0=[\sigma(x_i\partial_j),[x_i\partial_l,x_l\partial_k]]=[x+\lambda_i\partial_i+\lambda_j\partial_j,x_i\partial_k]=\lambda_i\partial_k$$
implies $\lambda_i=0.$
Therefore,
\begin{eqnarray}\label{E35-5}
\sigma(x_i\partial_j)=x_i\partial_j+\lambda_j\partial_j.
\end{eqnarray}
On one hand,
\begin{eqnarray}\label{E35-6}
[\sigma(x_i\partial_j),[x_j\partial_i,x_i\partial_k]]=[x_i\partial_j+\lambda_j\partial_j,x_j\partial_k]=x_i\partial_k+\lambda_j\partial_k.
\end{eqnarray}
On the other hand, using (\ref{e2}) and (\ref{E35-5}), one can derive
\begin{eqnarray}\label{E35-7}
[\sigma(x_i\partial_j),[x_j\partial_i,x_i\partial_k]]=-[\sigma(x_i\partial_k),x_i\partial_i-x_j\partial_j]=x_i\partial_k.
\end{eqnarray}
Comparing (\ref{E35-6}) with (\ref{E35-7}), one gets
$\lambda_j=0.$ By (\ref{E35-5}), we have proved $\sigma(x_i\partial_j)=x_i\partial_j.$

$\mathbf{Case\ 2:}$ $\;\mathfrak{g}=E(3,6)$ or $E(3,8).\;$ Firstly, let us show $\sigma(x_i\partial_j)=x_i\partial_j$ for any $x_i\partial_j\in\mathfrak{g}_0.$ Similar to the above, one may suppose
$$\sigma(x_i\partial_j)=x_i\partial_j+\sum_{m=1}^3\lambda_m\partial_m,\quad \lambda_m\in \mathbb{F}.$$

$\mathbf{Case\ 2.1:}$ $\;$  Let $i,j=1,2,3.$   For distinct $i,j,k,$ by (\ref{e2}),
$$0=[\sigma(x_i\partial_j),[ x_i\partial_k,x_k\partial_j]]=[\sigma(x_i\partial_j), x_i\partial_j]=
\left[x_i\partial_j+\sum_{m=1}^3\lambda_m\partial_m,x_i\partial_j\right]=\lambda_i\partial_j.$$
Consequently, $\lambda_i=0.$ One may suppose $\sigma(x_k\partial_j)=x_k\partial_j+\sum_{m\neq k}^3\mu_m\partial_m,$
On one hand,
\begin{eqnarray*}\label{E3-19}
[\sigma(x_i\partial_j), x_k\partial_i]=\left[x_i\partial_j+\sum_{m\neq i}\lambda_m\partial_m, x_k\partial_i\right]=-x_k\partial_j+\lambda_k\partial_i.
\end{eqnarray*}
On the other hand, by (\ref{e2}), one deduce
\begin{eqnarray*}
[\sigma(x_i\partial_j), x_k\partial_i]&=&[\sigma(x_i\partial_j),[ x_k\partial_j, x_j\partial_i]]\\\nonumber
&=&-[\sigma(x_k\partial_j),[ x_j\partial_i, x_i\partial_j]]-[\sigma(x_j\partial_i),[ x_i\partial_j, x_k\partial_j]]\\\nonumber
&=&-x_k\partial_j-\mu_j\partial_j+\mu_i\partial_i.\label{E3-20}
\end{eqnarray*}
Comparing the two equations above, one gets $\mu_j=0$ and $\lambda_k=\mu_i.$ In view of the arbitrariness of  $i,j,k,$ one may assume
$$\sigma(x_i\partial_j)=x_i\partial_j+\lambda\partial_k,\;\sigma(x_k\partial_j)=x_k\partial_j+\lambda\partial_i,\quad i,j,k \; \mbox{are distinct}.$$
Then
\begin{eqnarray*}
\sigma(x_i\partial_j)=\sigma[x_i\partial_k,x_k\partial_j]=[\sigma(x_i\partial_k),\sigma(x_k\partial_j)]
=x_i\partial_j-\lambda\partial_k.
\end{eqnarray*}
The two equations above imply  $\lambda=0,$ and then $\sigma(x_{i}\partial_{j})=x_i\partial_j,$ $i,j=1,2,3.$

$\mathbf{Case\ 2.2:}$ $\;$ Let $i,j=4,5.$  For distinct $k,l,q=1,2,3,$
$$0=[\sigma(x_i\partial_j),[x_k\partial_q,x_q\partial_l]]=[\sigma(x_i\partial_j),x_k\partial_l]=
\left[x_i\partial_j+\sum_{m=1}^3\lambda_m\partial_m,x_k\partial_l\right]=
\lambda_k\partial_l.$$
Therefor, $\lambda_k=0.$ The arbitrariness of $k$ implies $\sigma(x_i\partial_j)=x_i\partial_j.$

Next, we will prove $\sigma(h_m)=h_m$  for $m=1,2,3,4.$ Similarly,  suppose
$\sigma(h_m)=h_m+\sum_{k=1}^3\lambda_k\partial_k.$ Noting that $h_m$ is an element of the basis of  $\mathfrak{g}_0$'s Cartan subalgebra, one may assume $[h_m,x_i\partial_j]=\gamma x_i\partial_j,$ $\gamma\in\mathbb{F},$ $i,j=1,2,3.$
By the result obtained in Case 2.1, one can deduce
\begin{eqnarray*}
\sigma[h_m,x_i\partial_j]=\gamma\sigma(x_i\partial_j)=\gamma x_i\partial_j.
\end{eqnarray*}
On the other hand,
 \begin{eqnarray*}
\sigma[h_m,x_i\partial_j]=[\sigma(h_m),\sigma(x_i\partial_j)]=\left[h_m+\sum_{k=1}^3\lambda_k\partial_k,x_i\partial_j\right]
=\gamma x_i\partial_j+\lambda_i\partial_j.
\end{eqnarray*}
The two equations above imply that $\lambda_i=0$ for $i=1,2,3.$ Thus, we proved $\sigma(h_m)=h_m$  for $m=1,2,3,4.$

Summing up, we have proved $\sigma(x)=x$ for any $x\in \mathfrak{g}_0,$ that is, $\sigma|_{\mathfrak{g}_{0}}=\mathrm{id}|_{\mathfrak{g}_{0}}.$
\end{proof}

\subsection{ E(1,6)}
The exceptional simple Lie superalgebra $E(1,6)$ is a subalgebra of  $K(1,6).$ The principal gradation over $K(1,6)$(see \cite{SK}) induces an  irreducible consistent $\mathbb{Z}$-gradation over $E(1,6).$ Moreover,  $E(1,6)$ has the same non-positive $\mathbb{Z}$-graded components as $K(1,6)$:
$E(1,6)_j=K(1,6)_j,$ $j\leq 0.$ Let $x$ be even and $\xi_i,$ $i=1,\ldots,6,$ be odd indeterminates. Then
\begin{eqnarray*}
&&E(1,6)_0 \simeq \mathfrak{cspo}(6)=\langle x, \xi_i\xi_j \;|\;i,j=1,\ldots, 6, i\neq j\rangle,\\
&&E(1,6)_{-1}\simeq \mathbb{F}^6=\langle \xi _{i}\;|\;i=1,\ldots, 6\rangle,\\
&&E(1,6)_{-2}\simeq  \mathbb{F}.
\end{eqnarray*}
For $f,g\in E(1,6),$ the bracket product is defined as
\begin{eqnarray*}
[f,g]&=&\left(2f-\sum_{i=1}^6\xi_i\partial_i(f)\right)\partial_x(g)-(-1)^{|f||g|}\left(2g-\sum_{i=1}^6\xi_i\partial_i(g)\right)\partial_x(f)\\
&&+(-1)^{|f|}\sum_{i=1}^6\partial_i(f)\partial_i(g).
\end{eqnarray*}
\begin{proposition}\label{E6-p1}
If $\sigma$ is a   Hom-Lie superalgebra structure on $E(1,6),$ then $$\sigma|_{E(1,6)_{-1}}=\mathrm{id}|_{E(1,6)_{-1.}}$$
\end{proposition}
\begin{proof}
For $i,j,k=1,\ldots,6$ and $j\neq i,k,$ by (\ref{e2}),
\begin{eqnarray*}
[\sigma(\xi_i),\xi_k]&=&[\sigma(\xi_i),[\xi_j\xi_k,\xi_j]]=\delta_{k,i}[\sigma(\xi_j),\xi_j].
\end{eqnarray*}
It implies that $[\sigma(\xi_i),\xi_k]=0$ for any $k,i=1,\ldots,6.$  From $|\sigma|=\bar{0}$ and the transitivity of $E(1,6),$ one can deduce $\sigma(E(1,6)_{-1})\subseteq E(1,6)_{-1}$  So one may assume $\sigma(\xi_i)=\sum_{k=1}^6\lambda_k\xi_k,$ $\lambda_k\in \mathbb{F}.$ By using Equation (\ref{e2}), for distinct $i,j,s,t=1,\ldots,6$, one can deduce
\begin{eqnarray*}
0=[\sigma(\xi_i),[\xi_j\xi_t,\xi_s\xi_t]]=[\sigma(\xi_i),\xi_j\xi_s]
=\lambda_j\xi_s-\lambda_s\xi_j.
\end{eqnarray*}
Then $\lambda_k=0$ for any $k\neq i.$ That is $\sigma(\xi_i)=\lambda_i\xi_i.$ By using Equation (\ref{e2}) again, one can deduce
$$-\lambda_j=[\sigma(\xi_i),[\xi_j\xi_i,\xi_j]]=[\sigma(\xi_i),\xi_i]=[\lambda_i\xi_i,\xi_i]=-\lambda_i.$$
 Moreover, one knows  $\sigma(\xi_i)=\lambda\xi_i$ for any $i=1,\ldots,6.$
The remaining work is to prove
 $\lambda=1.$ For distinct $i,j,t,$
\begin{eqnarray}\label{6-1}
[\sigma(\xi_t\xi_j),\xi_t]=[\sigma(\xi_t\xi_j),[\xi_i\xi_t,\xi_i]]
=-[\sigma(\xi_i),[\xi_t\xi_j,\xi_i\xi_t]]
=[\lambda\xi_i,\xi_j\xi_i]=\lambda\xi_j.
\end{eqnarray}
Suppose $\sigma^{-1}$ is a left inverse of the monomorphism $\sigma,$ then
$$\sigma^{-1}([\sigma(\xi_t\xi_j),\xi_t])=[\xi_t\xi_j, \sigma^{-1}(\xi_t)]=\lambda^{-1}\xi_j.$$
Hence
\begin{eqnarray}\label{6-2}
[\sigma(\xi_t\xi_j),\xi_t]=\sigma(\lambda^{-1}\xi_j)=\xi_j.
\end{eqnarray}
Equations (\ref{6-1}) and (\ref{6-2}) imply $\lambda=1.$ The proof is completed.
\end{proof}
\begin{proposition}\label{E6-p2}
If $\sigma$ is a   Hom-Lie superalgebra structure on $E(1,6),$ then $$\sigma|_{E(1,6)_{0}}=\mathrm{id}|_{E(1,6)_{0.}}$$
\end{proposition}
\begin{proof}
For any $y\in {E(1,6)_{0}},$ by Proposition \ref{E6-p1} and Lemma \ref{00},
 $\sigma(y)-y\in E(1,6)_{-2}\simeq\mathbb{F}.$
So, one may suppose $\sigma(y)=y+\lambda,$ $\lambda\in \mathbb{F}.$ Then
\begin{eqnarray}\label{6-3}
[\sigma(y),x]=[y+\lambda,x]=[\lambda,x]=2\lambda.
\end{eqnarray}

If $y=x,$
noting that $x$ is an element of a basis of $E(1,6)_0$'s Cartan subalgebra, one may assume
\begin{eqnarray}\label{6-4}
[\sigma(x),x]=\mu\sigma(x)=\mu(x+\lambda),\quad \mu\in \mathbb{F}.
\end{eqnarray}
Comparing (\ref{6-3}) and (\ref{6-4}), one has $\mu=0,$ and then $\lambda=0,$ that is $\sigma(x)=x.$

If $y=\xi_i\xi_j,$
\begin{eqnarray}\label{6-5}
[\sigma(\xi_i\xi_j),x]=[\sigma(\xi_i\xi_j),\sigma(x)]=\sigma[\xi_i\xi_j,x]=0.
\end{eqnarray}
From (\ref{6-3}) and (\ref{6-5}), one gets $\lambda=0$ again. The proof is completed.
\end{proof}

\subsection{Result}
Propositions \ref{p4-1}, \ref{p4-2}, \ref{E5-p1}-\ref{E6-p2} show that
 the Hom-Lie superalgebra structures on the $0$-$\mathrm{th}$ and $-1$-$\mathrm{st}$ $\mathbb{Z}$-components of each infinite-dimensional  exceptional simple Lie superalgebras are trivial.
Together with Lemma  \ref{l1}, it is sufficient to draw the following conclusion.
\begin{theorem}
There is only the trivial  Hom-Lie superalgebra structure on each exceptional simple Lie superalgebras.
\end{theorem}


\begin{thebibliography}{99}
\bibitem{hls} J. T. Hartwig, D. Larsson and  S. D. Silvestrov.
Deformations of the Lie  algebras using $\sigma$-derivations.
\textit{J.  Algebra,} \textbf{295}  (2006): 314--361.



\bibitem{M2}  F. Ammar, Z. Ejbehi and A. Makhlouf. Cohomology and deformations of Hom-algebras. \textit{J. Lie Theory,}
\textbf{21} (2011): 813--836.



\bibitem{ls1}  A. Makhlouf, S. D. Silvestrov. Notes on 1-parameter formal deformations of Hom-associative and Hom-Lie algebras.
 \textit{Forum Math,} \textbf{22}  (2010): 715--739.

\bibitem{ls2}  D. Larsson, S. D. Silvestrov.
Quasi-deformations of $\frak{sl}_{2}(\mathbb{F})$  using twisted
derivations. \textit{Comm. Algebra,} \textbf{35} (2007):
4303--4318.

\bibitem{ls3} D. Yau. Hom-quantum groups: I. Quasi-triangular Hom-bialgebras.
 \textit{J. Phys. A: Math. Theor.,}
\textbf{45} (2012) 065203, 23 pages.

\bibitem{2016} L. Q. Cai, Y. H. Sheng.
Hom-Big Brackets: Theory and Applications. \textit{SIGMA}
\textbf{12} (2016): 18 pages.

\bibitem{s} Y. H. Sheng. Representations of Hom-Lie  algebras.  \textit{Algebr. Represent. Theory,} \textbf{15} (2012): 1081-1098.

\bibitem{bm}  S. Benayadi, A. Makhlouf. Hom-Lie  algebras with symmetric invariant nondegenerate bilinear forms.  \textit{J. Geom.
Phys.,}
\textbf{76}  (2014):
38--60.

\bibitem{jl} Q. Q. Jin, X. C. Li.
 Hom-Lie algebra structures on semi-simple Lie algebras. \textit{J. Algebra,}
\textbf{319}  (2008):
1398--1408.

\bibitem{xie} W. J. Xie,  Q. Q. Jin and W. D. Liu. Hom-structures on semi-simple Lie algebras.  \textit{Open Math,}
\textbf{13}  (2015): 617--630.

\bibitem{wuli1} N. Aizawa, H. Sato.  $q$-deformation of the Virasoro algebra with central extension. \textit{Phys. Lett. B,}
\textbf{256}  (1991): 185-190.

\bibitem{wuli2} M. Chaichian, P. Kulish and J. Lukierski. $q$-deformed Jacobi identity, $q$-oscillators and $q$-deformed infinite-dimensional algebras.
 \textit{Phys. Lett. B,}
\textbf{237}  (1990): 401-406.

\bibitem{wuli3} k. Liu. Characterizations of the quantum Witt algebra.
 \textit{Lett. Math. Phys.,}
\textbf{24}  (1992): 257-265.


\bibitem{hu} N. H. Hu, $q$-Witt algebras, $q$-Lie algebras, $q$-holomorph structure and representations.
\textit{Algebra Colloq.,} \textbf{6}  (1999): 51--70.

\bibitem{M1}  F. Ammar, A. Makhlouf. Hom-Lie superalgebras and Hom-Lie admissible superalgebras. \textit{J. Algebra,}
\textbf{324} (2010): 1513--1528.

\bibitem{cl} B. T. Cao, L. Luo.  Hom-Lie superalgebra structures on finite-dimensional simple Lie superalgebras. \textit{J. Lie Theory,}
\textbf{23} (2013): 1115--1128.

\bibitem{yuan2}  J. X. Yuan, W. D. Liu.  Hom-structures on finite-dimensional simple Lie superalgebras.
\textit{J. Math. Phys.,} \textbf{56} (2015) 061702.

\bibitem{K1}  V. G. Kac. Classification of infinite-dimensional
simple linearly compact Lie superalgebras. \textit{Adv. Math.,}
\textbf{139} (1998): 1--55.

\bibitem{yuan}  J. X. Yuan, L. P. Sun and W. D. Liu.  Hom-Lie superalgebra structures on infinite-dimensional simple Lie superalgebras of vector fields.
\textit{J. Geom. Phys.,} \textbf{84} (2014):1--7.





\bibitem{SK} S. J. Cheng, V. G. Kac. Structure of some $\mathbb{Z}$-graded Lie superalgebras of vector fields.  \textit{Transform. Groups,}
\textbf{4} (1999): 219--272.



\bibitem{five}  I. Shchepochkina. The five exceptional simple Lie superalgebras of vector fields. \textit{Funct. Anal. Appl.,}
\textbf{33} (1999): 208--219.

\bibitem{E36}  V. G. Kac, A. Rudakov. Representation of the exceptional Lie superalgebra E(3,6) I. Degeneracy Conditions. Shchepochkina.
\textit{Transform. Groups,} \textbf{7} (2002): 67--86.


\bibitem{E38}  V. G. Kac, A. Rudakov. Complexes of modules over exceptional Lie superalgebra E(3,8) and E(5,10) .
\textit{Int. Math. Res. Not. IMRN,} \textbf{19} (2002): 1007--1025.











\end{thebibliography}
\end{document}